\documentclass[a4paper,oneside,11pt, reqno, psamsfonts]{amsart}
\usepackage{graphicx}
\usepackage{float}
\usepackage{pstricks}
\usepackage{amscd}
\usepackage{amsmath}
\usepackage{amsxtra}
\usepackage[T1]{fontenc}
\usepackage[utf8]{inputenc}
\usepackage{lipsum}
\usepackage{tikz}
\usetikzlibrary{arrows}
\usetikzlibrary{calc}

\usepackage{amsfonts,amssymb,amsmath,amsthm}
\usepackage{url}
\usepackage{enumerate}

\newcommand{\strutstretchdef}{\newcommand{\strutstretch}{\vphantom{\raisebox{1pt}{$\big($}\raisebox{-1pt}{$\big($}}}}
\theoremstyle{plain}
\newtheorem{theorem}{Theorem}[section]

\newtheorem{lemma}[theorem]{Lemma}
\newtheorem{proposition}[theorem]{Proposition}
\newtheorem{corollary}[theorem]{Corollary}

\theoremstyle{definition}

\newtheorem{example}[theorem]{Example}

\theoremstyle{remark}
\newtheorem{remark}[theorem]{Remark}

\numberwithin{equation}{section}

\newlength{\struh}
\newlength{\textminustop}
\strutstretchdef
\newcommand*{\Ge}{\geqslant}

\makeatletter
\@namedef{subjclassname@2020}{%
  \textup{2020} Mathematics Subject Classification}
\makeatother

\usepackage{mathrsfs}
\usepackage{epsfig}
\usepackage[active]{srcltx}

\newcommand{\ncom}{\newcommand}
\ncom{\bq}{\begin{equation}}
\ncom{\eq}{\end{equation}}
\ncom{\beqn}{\begin{eqnarray*}}
\ncom{\eeqn}{\end{eqnarray*}}
\ncom{\beq}{\begin{eqnarray}}
\ncom{\eeq}{\end{eqnarray}}
\ncom{\nno}{\nonumber}
\ncom{\rar}{\rightarrow}
\ncom{\Rar}{\Rightarrow}
\ncom{\noin}{\noindent}
\ncom{\bc}{\begin{centre}}
\ncom{\ec}{\end{centre}}
\ncom{\sz}{\scriptsize}
\ncom{\rf}{\ref}
\ncom{\sgm}{\sigma}
\ncom{\Sgm}{\Sigma}
\ncom{\dt}{\delta}
\ncom{\Dt}{Delta}
\ncom{\lmd}{\lambda}
\ncom{\Lmd}{\Lambda}
\ncom{\eps}{\epsilon}
\ncom{\pcc}{\stackrel{P}{>}}
\ncom{\dist}{{\rm\,dist}}
\ncom{\sspan}{{\rm\,span}}
\ncom{\im}{{\rm Im\,}}
\ncom{\sgn}{{\rm sgn\,}}
\ncom{\ba}{\begin{array}}
\ncom{\ea}{\end{array}}
\ncom{\eop}{\hfill{{\rule{2.5mm}{2.5mm}}}}
\ncom{\eoe}{\hfill{{\rule{1.5mm}{1.5mm}}}}
\ncom{\eof}{\hfill{{\rule{1.5mm}{1.5mm}}}}
\ncom{\hone}{\mbox{\hspace{1em}}}
\ncom{\htwo}{\mbox{\hspace{2em}}}
\ncom{\hthree}{\mbox{\hspace{3em}}}
\ncom{\hfour}{\mbox{\hspace{4em}}}
\ncom{\hsev}{\mbox{\hspace{7em}}}
\ncom{\vone}{\vskip 2ex}
\ncom{\vtwo}{\vskip 4ex}
\ncom{\vonee}{\vskip 1.5ex}
\ncom{\vthree}{\vskip 6ex}
\ncom{\vfour}{\vspace*{8ex}}
\ncom{\norm}{\|\;\;\|}
\ncom{\integ}[4]{\int_{#1}^{#2}\,{#3}\,d{#4}}
\ncom{\inp}[2]{\langle{#1},\,{#2} \rangle}
\ncom{\Inp}[2]{\Langle{#1},\,{#2} \Langle}
\ncom{\vspan}[1]{{{\rm\,span}\#1 \}}}
\ncom{\dm}[1]{\displaystyle {#1}}

\begin{document}
\title[Analyticity, rank one perturbations and the invariance]{Analyticity, rank one perturbations and the invariance of the left spectrum}

\author[S. Chavan, S. Ghara and P. Pramanick]{Sameer Chavan, Soumitra Ghara and Paramita Pramanick}


\address{Department of Mathematics and Statistics\\
Indian Institute of Technology Kanpur, India}
   \email{chavan@iitk.ac.in}
 \email{ghara90@gmail.com}
   \address{School of Mathematics \\ Harish-Chandra Research Institute\\ 
Chhatnag Road, Jhunsi, Allahabad 211019, India}
\email{paramitapramanick@hri.res.in}


\thanks{The work of the second author is supported by The Fields Institute for Research in Mathematical Sciences. The third author is supported by the postdoctoral fellowship of Harish-Chandra Research Institute, Allahabad.}



\begin{abstract} 
We address the question of the analyticity of a rank one perturbation  of an analytic operator. If $\mathscr M_z$ is the bounded operator of multiplication by $z$ on a functional Hilbert space $\mathscr H_\kappa$ and $f \in \mathscr H$ with $f(0)=0,$ then 
$\mathscr M_z + f \otimes 1$ is always analytic. If $f(0) \neq 0,$ then 
the analyticity of $\mathscr M_z + f \otimes 1$ is characterized in terms of the membership to $\mathscr H_\kappa$ of the formal power series obtained by multiplying $f(z)$ by $\frac{1}{f(0)-z}.$  As an application, we discuss the  problem of the invariance of the left spectrum under rank one perturbation. In particular, we show that the left spectrum $\sigma_l(T + f \otimes g)$ of the rank one perturbation $T + f \otimes g,$ $\,g \in \ker(T^*),$  of a cyclic analytic left invertible bounded linear operator $T$ coincides with the left spectrum of $T$ except the point $\inp{f}{g}.$ In general, the point $\inp{f}{g}$ may or may not belong to $\sigma_l(T + f \otimes g).$  
However, if it belongs to $\sigma_l(T + f \otimes g) \backslash \{0\},$ then it is a simple eigenvalue of $T + f \otimes g.$
\end{abstract}

\maketitle

\section{Analyticity and the invariance of the left spectrum}

Several examples of multiplication operators on reproducing kernel Hilbert spaces suggest that the essential spectrum and the left spectrum coincide. Since the essential spectrum of a bounded linear operator is always invariant under compact perturbations (see \cite{CM}), it is reasonable to ask whether the left spectrum is also invariant under a compact or a finite rank perturbation ? On a testing ground, we address the above problem for rank one perturbations.  The main result of this note exploits the Shimorin's analytic model (see \cite{Shimorin-2001}) for analytic left invertible operators to provide a solution to this problem when the operator in question is a cyclic analytic left-invertible operator and the perturbation is of the form $f \otimes g$ with $g \in \ker T^*.$ We see that solution to the above problem is closely related to the description of the hyper-range of the rank one perturbation of analytic operators. 
Before we state the main result, we fix some notations and collect the necessary preliminaries.

Let $\mathbb C$ denote the set of complex numbers. For a positive real number $r,$ let $\mathbb D_r$ denote the open disc centered at $0$ and of radius $r.$ For simplicity, the open unit disc $\mathbb D_1$ will be denoted by $\mathbb D.$ For any holomorphic function $\phi : \mathbb D \rar \mathbb C$ and an integer $n \geqslant 0,$ let $\hat{\phi}(n)$ denote the coefficient of $z^n$ in the power series representation of $\phi.$ Let $\overline{\partial}$ denote the partial derivative with respect to $\overline{w}.$ 
Let $\mathcal H$ be a nonzero complex Hilbert space and 
let $T$ be a bounded linear operator on $\mathcal H.$ Let $\mathcal B(\mathcal H)$ denote the set of bounded linear operators on $\mathcal H.$ Recall that $\mathcal B(\mathcal H)$ is a unital $C^*$-algebra with unit $I.$  
We denote by $\sigma_p(T),$ $\sigma(T),$ $\sigma_l(T)$ and $\sigma_e(T)$ the point spectrum, spectrum, left spectrum and essential spectrum of the bounded linear operator $T,$ respectively. The spectral radius of $T$ is denoted by $r(T)$ (refer to \cite{CM} for definitions and basic spectral theory).
For $f, g \in \mathcal H,$ the bounded linear operator $f \otimes g$ on $\mathcal H$ is given by
\beqn
f \otimes g(h) = \inp{h}{g} f, \quad h \in \mathcal H.
\eeqn
A bounded linear operator $T$ on $\mathcal H$ is {\it left invertible}
if there exists a bounded linear operator $L$ on $\mathcal H$ (a {\it left-inverse}) such that $LT=I.$
For a positive integer $m,$ an operator $T$ on $\mathcal H$ is said
to be {\it $m$-cyclic} if there is an
$m$-dimensional vector subspace $\mathcal{M}$ of
$\mathcal{H}$, called the {\em cyclic space} of
$T$, such that
 $
\mathcal H$ is the closed linear span of $\{T^n h : n \Ge 0, h \in \mathcal M\}.$ For simplicity, we refer $1$-cyclic operator as the {\it cyclic} operator. 
   A bounded linear operator $T$ on $\mathcal H$ is {\it
analytic} if the {\it hyper-range} $\bigcap_{n=0}^{\infty}{T^n}\mathcal H$ of $T$ is $\{0\}.$ We
say that $T$ on $\mathcal H$ has the {\it
wandering subspace property} if
   \begin{align*}
\mathcal H = \bigvee_{n \Ge 0} T^n(\ker T^*).
   \end{align*}

Let $\mathscr H_\kappa$ be a reproducing kernel Hilbert space of complex-valued holomorphic functions defined on the unit disc and let $\kappa : \mathbb D \times \mathbb D \rar \mathbb C$ be the reproducing kernel for $\mathscr H_\kappa,$ that is, $\kappa(\cdot, w) \in \mathscr H_\kappa$ and 
\beqn
\inp{f}{\kappa(\cdot, w)}_{\mathscr H_\kappa} = f(w), \quad f \in \mathscr H_\kappa, ~w \in \mathbb D.
\eeqn
The reader is referred to \cite{PR} for the basics of reproducing kernel Hilbert spaces.
We call $\mathscr H_\kappa$ a {\it functional Hilbert space} if the following hold:
\begin{enumerate}
\item[(A1)] if $h : \mathbb D \rar \mathbb C$ is a holomorphic function, then $h \in \mathscr H_\kappa$ if and only if $zh \in \mathscr H_\kappa,$
\item[(A2)] $\kappa$ is normalized at the origin, that is, $\kappa(z, 0)=1$ for every $z \in \mathbb D,$
\item[(A3)] Under the assumptions (A1) and (A2), the orthogonal complement of $\{zf : f \in \mathscr H_\kappa\}$ is spanned by the space of constant functions.
\end{enumerate} 
\begin{remark} \label{rkhs}
Assume that $\mathscr H_\kappa$ is a functional Hilbert space. 
By (A1) and the closed graph theorem, the operator $\mathscr M_z$ of multiplication by the coordinate function $z$ defines a bounded linear operator on $\mathscr H_\kappa.$ By (A2),
$\mathscr H_\kappa$ consists of constant functions, and hence by (A1), 
$\mathscr H_\kappa$ contains the linear space of polynomials. 
By (A3), the kernel of $\mathscr M^*_z$ is spanned by the constant function~$1.$ 
\end{remark}
   
Here is the main result of this paper (cf. \cite[Proposition~4.2]{JK}).
\begin{theorem} \label{conse-proof}
Let $\mathscr H_\kappa$ be a functional Hilbert space and and let $f \in \mathscr H_\kappa.$
Let $\mathscr M_z$ denote the operator of multiplication by $z$ on $\mathscr H_\kappa.$  
Then  $\mathscr M_z+f \otimes 1$ is analytic if and only if exactly one of the following holds$:$
\begin{enumerate}
\item[$(i)$] $f(0)=0,$ 
\item[$(ii)$] $f(0) \neq 0$ and $\displaystyle \sum_{j=0}^{\infty} \Big(\sum_{i=0}^{j}   \frac{\hat{f}(j-i)}{f(0)^{i}}\Big) z^j~\mbox{does not belong to}~ \mathscr H_\kappa,$ where $\hat{f}(n)$ denotes the coefficient of $z^n$ in the power series representation of $f.$
\end{enumerate}
\end{theorem}
\begin{remark}
Note that the formal power series in (ii) above is nothing but the product of $f$ with the formal power series $\displaystyle \sum_{j=0}^{\infty} \Big(\frac{z}{f(0)}\Big)^j.$
\end{remark}

Although Theorem~\ref{conse-proof} is of independent interest, it helps understand the problem of the invariance of the left spectra under compact perturbations. To see some general facts pertaining to this problem,  let $T, K$ be bounded linear operators on $\mathcal H.$ If $K$ is a compact operator, then 
\beq \label{two-incl}
\sigma_l(T+K) \setminus \sigma_p(T+K)  \subseteq \sigma_l(T), \quad \sigma_l(T) \setminus \sigma_p(T)  \subseteq \sigma_l(T+K).
\eeq  
Indeed, if $\lambda \notin \sigma_l(T),$ then 
\beqn
&& \mbox{$T-\lambda$ is left-invertible} \Rightarrow \mbox{$T+K-\lambda$ is left semi-Fredholm} \\ 
&& \Rightarrow \mbox{$\mbox{ran}(T+K-\lambda)$ is closed and $\ker(T+K-\lambda)$ is finite dimensional},
\eeqn
and hence either $\lambda \in \mathbb C \backslash \sigma_l(T+K)$ or $\lambda \in \sigma_p(T+K)$  completing the verification of the first inclusion in \eqref{two-incl}. The second inclusion in \eqref{two-incl} can be obtained along similar lines. The obvious question appears here is when equality occurs in these inclusions. 
This plus little more yields the following:
\begin{proposition} \label{specrtum of perturbation}
Let $T, K \in \mathcal B(\mathcal H).$ Assume that $K$ is a compact operator. Then the following statements are valid$:$
\begin{enumerate}
\item[$(i)$] if $\sigma_p(T)=\emptyset,$ then $\sigma_l(T+K) = \sigma_l(T) \cup \sigma_p(T+K),$
\item[$(ii)$] for every $\lambda \in \sigma_p(T+K) \backslash \sigma_l(T),$ $\ker(T+ K - \lambda)$ is finite dimensional.
\end{enumerate} 
\end{proposition}
\begin{proof}  
If $\sigma_p(T)=\emptyset,$ then by \eqref{two-incl}, 
\beqn
\sigma_l(T) = \sigma_l(T) \setminus \sigma_p(T)  \subseteq \sigma_l(T+K) \subseteq   \sigma_l(T) \cup \sigma_p(T+K),
\eeqn
which yields (i). To see (ii), let $\lambda \in \mathbb C \backslash \sigma_l(T).$ For $h \in \mathcal H,$ note that 
 \beq \label{iff-eigen-new}
\mbox{ $(T+ K -\lambda)h=0$ if and only if $(T-\lambda)h =-Kh.$} 
\eeq
Applying a left-inverse $L_\lambda$ of $T-\lambda$ on both sides, we get
\beq \label{formula-h-new}
h =- L_\lambda K h. 
\eeq
Thus the dimension of $\ker(T+ K - \lambda)$ is the dimension of $\ker(L_\lambda K +I).$ Since $L_\lambda K$ is a compact operator, $\ker(T+ K - \lambda)$ is finite dimensional (see \cite[Corollary~1.2.3 and Theorem~1.3.3]{CM}). 
\end{proof}
Thus the problem of invariance of the left-spectra could be answered provided we determine  eigenvalues of the compact perturbation. Let us analyse this problem for $K=f \otimes g$ with $g \in \ker T^*.$ 
   \begin{proposition} 
\label{kernel-one} Let $T \in \mathcal B(\mathcal H)$ and $f, g \in \mathcal H.$ For $\lambda \notin \sigma_l(T),$ the following statements are valid$:$ 
\begin{enumerate}
\item[$(i)$] the dimension of $\ker(T+f \otimes g - \lambda)$ is at most $1,$
\item[$(ii)$] if $g \in \ker T^*$ and $\lambda \neq \inp{f}{g},$ then $\lambda \notin \sigma_p(T+f\otimes g).$ 
\end{enumerate}
 In particular, if $g \in \ker T^*$ and $\sigma_p(T) = \emptyset,$ then $$\sigma_p(T+f\otimes g) \subseteq \sigma_l(T) \cup \{\inp{f}{g}\}.$$
 \end{proposition}
 \begin{proof}
 By \eqref{formula-h-new}, any $h \in \ker(T+f \otimes g - \lambda)$ is a multiple of $L_\lambda f,$ and hence (i) follows.
 If
$\lambda$ is an eigenvalue of $T+f\otimes g$ with eigenvector $h,$  then by \eqref{formula-h-new}, $\inp{h}{g} \neq 0,$ and hence 
by \eqref{iff-eigen-new}, $f=-\frac{(T-\lambda)h}{\inp{h}{g}}.$ Taking inner-product with $g,$ this implies that $\lambda = \inp{f}{g}$ (since $g \in \ker T^*$). 
This yields (ii).
To see the remaining part, note that 
 $\sigma_p(T+f\otimes g) \setminus \sigma_l(T) \subseteq \{\inp{f}{g}\}$ (by (ii)) and apply
Proposition~\ref{specrtum of perturbation}. 
 \end{proof}
 
Thus a solution to the problem of the invariance of the left spectra for the aforementioned rank one perturbation boils down to deciding whether or not $\inp{f}{g}$ is an eigenvalue of $T+f\otimes g.$ This lies deeper! 
Nevertheless,  
as an application of the proof of Theorem~\ref{conse-proof}, we obtain the following$:$ 
\begin{theorem} \label{main}
Let $T$ be a cyclic analytic left invertible operator in $\mathcal B(\mathcal H).$ Assume that $f \in \mathcal H$ and $g \in \ker T^*.$ Then $\sigma_l(T)$ is a subset of $\sigma_l(T + f \otimes g)$ such that
\beqn
&& \sigma_l(T + f \otimes g) \backslash \{\inp{f}{g}\}  =  \sigma_l(T) \backslash \{\inp{f}{g}\}, \\ && r(T + f \otimes g) = \max\{r(T), |\inp{f}{g}|\}.
\eeqn
Moreover, the following statements are valid$:$
\begin{enumerate}
\item[$(i)$] if $\inp{f}{g}$ belongs to $\sigma_l(T + f \otimes g) \backslash \{0\},$ then it is a simple eigenvalue of $T + f \otimes g,$
\item[$(ii)$] either $\sigma_l(T + f \otimes g) = \sigma_l(T)$ or $\sigma_l(T + f \otimes g)= \sigma_l(T) \cup \{\inp{f}{g}\}.$
\end{enumerate}
 \end{theorem}
\begin{remark}In general, the inclusion $\sigma_l(T)  \subseteq \sigma_l(T + f \otimes g)$ and the inequality $r(T + f \otimes g) \Ge r(T)$  are strict.  Further, any of following possibilities can occur: $r(T + f \otimes g) < |\inp{f}{g}|,$ $r(T+ f \otimes g) = |\inp{f}{g}|,$ $r(T+ f \otimes g) > |\inp{f}{g}|.$  
\end{remark}


The proofs of Theorems~\ref{conse-proof} and \ref{main} occupy major portion of Section~2. In the remaining part of this section, we explain the essential difference between the problems of the invariance of the essential spectra and that of left spectra with the help of one concrete family of multiplication operators.  Recall that the Fredholm index is invariant under compact perturbations (see \cite[Theorem~1.3.1]{CM}). So if one can show that the dimension of the cokernel is not preserved under perturbations, then so is the dimension of the kernel, and hence the left spectrum may not be invariant under compact perturbations.
The following proposition supports these speculations (cf. \cite[Theorem~1]{St}).
\begin{proposition} 
Let $\mathscr H_\kappa$ be a functional Hilbert space and and let $\mathscr M_z$ denote the operator of multiplication by $z.$  
Assume that 
the kernel $\kappa$ satisfies 
\beq \label{sharp}
\ker(\mathscr M^*_z-\overline{w}) = \{\alpha \kappa(\cdot, w) : \alpha \in \mathbb C\}, \quad w \in \mathbb D.
\eeq
If $f \in \mathscr H_\kappa$ and $S := \mathscr M_z + f \otimes 1,$ 
 then the following statements are valid$:$
\begin{enumerate}
\item[$(i)$]  if $f(0) = 0,$ then  
\beqn 
\ker S^* =
\begin{cases} 
 \mbox{span}\{\kappa(\cdot, 0), \overline{\partial} \kappa(\cdot, w)|_{w=0}\} & \mbox{if}~f'(0)=-1, \\
\mbox{span}\{\kappa(\cdot, 0)\} & \mbox{otherwise},
\end{cases}
\eeqn
\item[$(ii)$]  if $f(0) \neq 0,$ then  $$\ker S^*=\mbox{span}\{ (1 + \overline{f'(0)})\kappa(\cdot, 0) - \overline{f(0)} \, \overline{\partial} \kappa(\cdot, w)|_{w=0}\}.$$ 
\end{enumerate}
In particular, if $f(0)=0$ and $f'(0) = -1,$ then $S$ is not cyclic. 
\end{proposition}
\begin{proof} 
Recall from
\cite[Lemma~4.1]{CS} and \cite[Lemma 1.22(ii)]{CD} that 
\beq 
 \left.
 \begin{array}{cc}
& e_n := \frac{\overline{\partial}^n \kappa(\cdot, w)}{n!}\Big|_{w=0} \in \mathscr H_\kappa, ~\inp{f}{e_n} = \frac{({\partial}^nf)(0)}{n!}, \quad f \in \mathscr H_\kappa, ~n \Ge 0, \\ \label{ip-with-en}
& \mathscr M^*_ze_n = \begin{cases} 0 & \mbox{if}~n=0, \\
e_{n-1} & \mbox{otherwise}.
\end{cases}
 \end{array}
\right\}
\eeq
We claim that 
\begin{equation}\label{eqnkerS*}
\ker S^*=\{ae_0+be_1: b(1+\overline{f'(0)})+a\overline{f(0)}=0\}.
\end{equation}
To see this, note that since $\kappa(\cdot,0)=1,$ by \eqref{ip-with-en}, 
\begin{align}\label{eqnkerdim}
S^*(ae_0+be_1)&=M^*_z(a  e_0 + b e_1) + 1 \otimes f(a  e_0 + b e_1)\notag\\ &=b(1+\overline{f'(0)})+a\overline{f(0)}.
\end{align}
Thus, if $b(1+\overline{f'(0)})+a\overline{f(0)}=0$, then 
$ae_0+be_1\in \ker S^*.$ Conversely, let $h \in \ker S^*.$ Note that
$\mathscr M^*_zh = -\inp{h}{f},$
and since $1 \in \ker \mathscr M^*_z,$
$h$ belongs to $\ker \mathscr M^{*2}_z.$ 
Since $\kappa$ satisfies \eqref{sharp}, by \cite[Lemma 1.22(ii)]{CD}, $h \in \mbox{span}\big\{e_0, e_1\big\}.$ This combined with \eqref{eqnkerdim} 
yields \eqref{eqnkerS*}. 

Assume now that $f(0)=0.$  
Thus by \eqref{eqnkerS*},
\beqn
\ker S^*=\{ae_0+be_1: b(1+\overline{f'(0)})=0\}.
\eeqn
In case  $f'(0)=-1,$ $\ker S^*=\mbox{span}\{e_0, e_1\}.$ Otherwise,  $\ker S^*=\mbox{span}\{e_0\}.$
If $f(0) \neq 0,$ then by \eqref{eqnkerS*}, $\ker S^*$ consists of elements of the form $ae_0+be_1$ where $a = -\frac{b (1 + \overline{f'(0)})}{\overline{f(0)}}$ and  we get the desired conclusion in (ii). 
Since $\mathscr H_\kappa$ contains all complex polynomials, it is easy to see using \eqref{ip-with-en} that $\kappa(\cdot, 0)$ and $\overline{\partial} \kappa(\cdot, w)|_{w=0}$ are linearly independent.  
The remaining part now follows from (i) and \cite[Proposition~1(i)]{Herrero1978}.
\end{proof}


\allowdisplaybreaks

\section{Proofs of Theorems~\ref{main} and \ref{conse-proof}}

In the proof of Theorem~\ref{conse-proof}, we employ reproducing kernel techniques to describe the hyper-range of a rank one perturbation $f \otimes g,$ $g \in \ker \mathscr M^*_z,$ of $\mathscr M_z.$  This combined with 
the Shimorin's analytic model yields the proof of Theroem~\ref{main}.
Recall from
\cite[Pg~154] {Shimorin-2001} that  any analytic left-invertible operator $T$ is unitarily equivalent to the operator $\mathscr M_z$ of multiplication by $z$ on a reproducing kernel Hilbert space $\mathscr H_\kappa$ with kernel $\kappa.$  
Indeed, $\mathscr H_\kappa$ consists of $\ker T^*$-valued holomorphic functions on a disc centered at the origin in the complex plane and 
the reproducing kernel $\kappa$ of $\mathscr H_\kappa$ satisfies $\kappa(\cdot, 0)=1$ (see \cite[Eqn~(2.6)] {Shimorin-2001}). 
If, in addition, $T$ is cyclic, then $\ker T^*$ is one-dimensional, and hence, in the proof of Theorem~\ref{main}, we may assume that $\mathscr H_\kappa$ is a functional Hilbert space.
To prove Theorems~\ref{conse-proof} and \ref{main}, we need several lemmas. The first one is of algebraic nature and holds for any bounded linear operator on a complex Hilbert space.

\begin{lemma} \label{n-power}
Let $T \in \mathcal B(\mathcal H)$ and let $f, g \in \mathcal H.$ If $T^*g=0,$ then the following statements are valid$:$
\begin{enumerate}
\item[$(i)$] for any positive integer $n,$ 
\beq \label{above-f}
(T+f \otimes g)^n = T^n +\sum_{j=0}^{n-1}\inp{f}{g}^{n-j-1} T^{j}f \otimes g, \eeq
\item[$(ii)$] if $T$ is cyclic, then $T+f \otimes g$ is $2$-cyclic. 
\end{enumerate}
\end{lemma}
\begin{proof} Suppose that $T^*g=0.$

(i) Clearly, the formula \eqref{above-f} holds for $n=1.$ Assume the formula  \eqref{above-f} for a positive integer $n \Ge 2.$ Since $T^*g =0,$ $(f \otimes g) T = f \otimes T^*g=0$ and  for any $h \in \mathcal H,$ 
\beqn (f \otimes g)(T^j f \otimes g)h = \inp{h}{g} \inp{T^j f}{g}f
=0, \quad j \geqslant 1.
\eeqn 
This combined with the induction hypothesis yields 
\beqn (T + f \otimes g)(T+f \otimes g)^n &=& T^{n+1} +\sum_{j=0}^{n-1}\inp{f}{g}^{n-j-1} T^{j+1}f \otimes g \\ &+& (f \otimes g) T^n + \sum_{j=0}^{n-1}\inp{f}{g}^{n-j-1} (f \otimes g) \, T^{j}f \otimes g\\
&=& T^{n+1} +\sum_{j=0}^{n-1}\inp{f}{g}^{n-j-1} T^{j+1}f \otimes g  \\ &+& \inp{f}{g}^{n-1}  (f \otimes g)^2 \\
&=& T^{n+1} +\sum_{j=0}^{n}\inp{f}{g}^{n-j} T^{j}f \otimes g.
\eeqn
This completes the verification of (i).

(ii) Let $\xi \in \mathcal H.$ A routine verification using \eqref{above-f} shows that for any integer $n \Ge 1,$
\beqn
(T+f \otimes g)^n \xi &=& T^n\xi  + \inp{\xi}{g}  \sum_{j=0}^{n-1}\inp{f}{g}^{n-j-1} T^{j}f, \\
(T+f \otimes g)^{n-1} f &=& \sum_{j=0}^{n-1}\inp{f}{g}^{n-j-1} T^{j}f.
\eeqn
It is immediate that
\beqn
T^n \xi = (T+f \otimes g)^n \xi - \inp{\xi}{g} (T+f \otimes g)^{n-1} f.
\eeqn
Thus, if $\xi$ is a cyclic vector for $T,$ then $T+f \otimes g$ is $2$-cyclic with the set of cyclic vectors equal to $\{\xi, f\}.$
\end{proof}

The next lemma required in the proof of Theorem~\ref{main} describes the hyper-range of rank one perturbations of left invertible analytic operators (cf. \cite[Lemma~2.4]{ACT}).

\begin{lemma} \label{Soumitra}
For $r > 0,$ let $\mathscr H$ be a Hilbert space of complex-valued holomorphic functions on $\mathbb D_r$ such that $\mathscr H$ contains all complex polynomials in $z.$  
Assume that 
\beq \label{hypothesis}
\mbox{$h \in \mathscr H$ if and only if $zh \in \mathscr H.$}
\eeq
Let $\mathscr M_z$ denote the operator of multiplication by $z$ and let $f, g \in \mathscr H$ be such that $g \in \ker \mathscr M^*_z.$  
Then the following statements are valid$:$
\begin{enumerate}
\item[$(i)$] if $\inp{f}{g}=0,$ then $\mathscr M_z + f \otimes g$ is analytic,
\item[$(ii)$] if $f(0) \neq 0$ and $\inp{f}{g} \neq 0,$ then the hyper-range $\mathscr R_{\infty}$ of $\mathscr M_z + f \otimes g$ is given by 
\beqn
\mathscr R_{\infty} = \begin{cases} \displaystyle \mathrm{span}\Big\{h_0\Big\} & \mbox{if}~h_0~\mbox{belongs to} ~\mathscr H,\\
\{0\} & \mbox{otherwise},
\end{cases}
\eeqn
where 
$
h_0:=\displaystyle \sum_{j=0}^{\infty} \Big(\sum_{i=0}^{j}   \frac{\hat{f}(j-i)}{\inp{f}{g}^{i}}\Big) z^j,$ 
\item[$(iii)$] 
if $\inp{f}{g} \neq 0$ and $h_0$ converges in $\mathscr H,$ then $h_0$ is an eigenfunction of $\mathscr M_z + f \otimes g$ with respect to the simple eigenvalue $\inp{f}{g}.$
\end{enumerate}
\end{lemma}
\begin{proof} Let 
$\mathscr M_{f, g}:=\mathscr M_z + f \otimes g.$ Note that  
$h \in \mathscr R_{\infty}$ if and only if for every integer $n \geqslant 1,$ there exists $h_n \in \mathscr H$ such that
$
h=\mathscr M^n_{f, g}h_n.$ Thus, by Lemma \ref{n-power}(i),  $h \in \mathscr R_{\infty}$ if and only if there exists $\{h_n\}_{n \geqslant 1} \subseteq \mathscr H$ such that
\beq \label{ind}
h = z^n h_n + \inp{h_n}{g} \sum_{j=0}^{n-1} \inp{f}{g}^{n-j-1}z^{j}f, \quad n \geqslant 1.
\eeq

(i): In case $\inp{f}{g}=0,$ $h=z^n h_n + \inp{h_n}{g}z^{n-1}f$ 
for every positive integer $n,$ 
and hence $h$ has zero at $0$ of arbitrary order. Thus $h=0$ in this case. 

(ii):  Assume that $\hat{f}(0) \neq 0$ and $\inp{f}{g} \neq 0.$  Fix an integer $n \geqslant 1.$ Letting $f(z)=\sum_{k=0}^{\infty} \hat{f}(k) z^k$ in 
\eqref{ind}, we get 
\beqn
h = z^n h_n + \inp{h_n}{g} \sum_{k=0}^{\infty} \sum_{i=0}^{n-1} \hat{f}(k) \inp{f}{g}^{n-i-1}z^{i+k}, \quad n \geqslant 1.
\eeqn
Comparing the coefficients on the both sides, we obtain 
\beq \label{identity}
\hat{h}(j) = \inp{h_n}{g} \sum_{i=0}^{j}  \inp{f}{g}^{n-i-1} \hat{f}(j-i), \quad j=0, \ldots, n-1, ~n \Ge 1.
\eeq
Letting $j=0$ in the above identity, we get $\hat{h}(0)=  \inp{h_n}{g}\inp{f}{g}^{n-1} \hat{f}(0).$ Thus, \eqref{identity} simplifies to 
 \beqn
\hat{h}(j) = \frac{\hat{h}(0)}{\hat{f}(0)} \sum_{i=0}^{j}   \frac{\hat{f}(j-i)}{\inp{f}{g}^{i}}, \quad j=0, \ldots, n-1, ~n \Ge 1,
\eeqn
where we used the assumption that $f(0)=\hat{f}(0)$ and $\inp{f}{g}$ are nonzero. Since $n$ was arbitrary, we obtain that $h= \frac{\hat{h}(0)}{\hat{f}(0)}h_0.$ 
Thus,  if $h_0$ does not belong to $\mathscr H,$ then $\mathscr R_{\infty} = \{0\}.$

We now check that if $h_0 \in \mathscr H,$ then $\mathscr R_{\infty}$ is spanned by $h_0.$ 
To see this, assume that $h_0 \in \mathscr H.$ Note that $h_0$ belongs to the hyper-range of $\mathscr M_{f, g}$ provided there exists a sequence $\{h_n\}_{n \geqslant 1}$ in $\mathscr H$ such that \eqref{ind} holds for $h=h_0$.
It follows from the definition of $\hat{h}_0(\cdot)$ and \eqref{hypothesis} that $h_0 -  \sum_{j=0}^{n-1} \frac{z^j f}{\inp{f}{g}^{j}}$ is divisible by $z^n$ in $\mathscr H.$  
Thus there exists a sequence $\{h_n\}_{n \geqslant 1}$ in $\mathscr H$ such that 
\beq
\label{h-0}
h_0 -  \sum_{j=0}^{n-1} \frac{z^j f}{\inp{f}{g}^{j}}  = z^n h_n, \quad n \geqslant 1.
\eeq
Fix $n \geqslant 1.$
Comparing the coefficient of $z^n$ on both sides, we obtain $$\hat{h}_0(n) - \sum_{j=0}^{n-1} \frac{\hat{f}(n-j)}{\inp{f}{g}^{j}} = \hat{h}_n(0).$$
However, by the definition of $h_0,$ this simplifies to
\beq \label{2.6}
\frac{\hat{f}(0)}{\inp{f}{g}^{n}} = \hat{h}_n(0).
\eeq
Since any $\phi \in  \mathscr H$ can be written as $\phi=\hat{\phi}(0) + z \psi$ for $\psi \in \mathscr H$ (see \eqref{hypothesis}) and $g \in \ker \mathscr M^*_z,$  we have $\inp{\phi}{g}=\inp{\hat{\phi}(0)}{g},$ and hence  
\beqn
\inp{h_n}{g} = \inp{\hat{h}_n(0)}{g} \overset{\eqref{2.6}}= \frac{\inp{\hat{f}(0)}{g}}{\inp{f}{g}^n} = \frac{1}{\inp{f}{g}^{n-1}}.
\eeqn
Combining this with \eqref{h-0}, we see that $\{h_n\}_{n \geqslant 1}$ satisfies 
\beqn 
h_0 = z^n h_n + \inp{h_n}{g} \sum_{j=0}^{n-1} \inp{f}{g}^{n-j-1}z^{j}f, \quad n \geqslant 1.
\eeqn
This completes the proof of (ii).


(iii) Assume that $\inp{f}{g} \neq 0$ and $h_0 \in \mathscr H.$ By (ii), $\mathscr R_\infty$ is one-dimensional invariant subspace of $\mathscr M_z + f \otimes g$, and hence there exists $\lambda \in \mathbb C$ such that 
$(\mathscr M_z + f \otimes g)h_0 = \lambda h_0.$ 
Since $g \in \ker \mathscr M^*_z,$ 
\beqn 
&& (\mathscr M_z + f \otimes g)h_0 \notag \\ &=&  \notag \sum_{j=0}^{\infty} \Big(\sum_{i=0}^{j}   \frac{\hat{f}(j-i)}{\inp{f}{g}^{i}}\Big) z^{j+1} + \sum_{j=0}^{\infty} \Big(\sum_{i=0}^{j}   \frac{\hat{f}(j-i)}{\inp{f}{g}^{i}}\Big) \inp{z^j}{g}  f \\
&=& \sum_{j=0}^{\infty} \Big(\sum_{i=0}^{j}   \frac{\hat{f}(j-i)}{\inp{f}{g}^{i}}\Big) z^{j+1} + \inp{f}{g} \sum_{k=0}^{\infty} \hat{f}(k)z^k \\
&=& \lambda \sum_{j=0}^{\infty} \Big(\sum_{i=0}^{j}   \frac{\hat{f}(j-i)}{\inp{f}{g}^{i}}\Big) z^{j}. \notag 
\eeqn
Comparing the constant terms on both sides, we get $\lambda =\inp{f}{g}.$  Since $\mathscr R_\infty$ contains eigenfunction corresponding to any nonzero eigenvalue of $\mathscr M_z + f \otimes g$ and since $\mathscr R_\infty$ is one dimensional, $\inp{f}{g}$ is a simple eigenvalue of $\mathscr M_z + f \otimes g,$ completing the proof. 
\end{proof}

 \begin{proof}[Proof of Theorem~\ref{conse-proof}] Since $\mathscr H_\kappa$ is a functional Hilbert space,  $\inp{f}{1}=f(0).$ The desired conclusion is now immediate from Lemma~\ref{Soumitra} (with $g=1$). \end{proof}

Here is a particular case of Theorem~\ref{conse-proof}.
\begin{corollary} 
\label{coro-main}
Assume  the hypotheses of Theorem~\ref{conse-proof}.  If $f$ is a polynomial of degree $n,$ then the operator $\mathscr M_z + f \otimes 1$ is analytic if and only if exactly one of the following holds$:$
\begin{enumerate}
\item[$(i)$] $f(0) =0,$  
\item[$(i)$] $f(0) \neq 0,$ $f(f(0)) \neq 0$ and $\frac{1}{f(0)-z} \notin \mathscr H_\kappa.$
\end{enumerate}
\end{corollary}
\begin{proof}
Assume that $f(z) = \sum_{k=0}^n \hat{f}(k)z^k,$ $z \in \mathbb D,$ and $f(0) \neq 0.$ 
Since $\mathscr H_\kappa$ contains all polynomials (see Remark~\ref{rkhs}), 
\beq \label{equi-h0-h1}
\mbox{$h_0 \in \mathscr H_\kappa$ $\Leftrightarrow$ $h_1 := \sum_{j=n}^{\infty} \Big(\sum_{i=0}^{j}   \frac{\hat{f}(j-i)}{f(0)^{i}}\Big) z^j \in \mathscr H_\kappa.$}
\eeq 
Also, since $\hat{f}(j)=0$ for $j > n,$
\beq \label{h1-equi} \notag 
h_1 &=& \sum_{j=n}^{\infty} \Big(\sum_{i=j-n}^{j}   \frac{\hat{f}(j-i)}{f(0)^{i}}\Big) z^j \\
&=& \notag
\Big(\sum_{k=0}^{n}  \hat{f}(k) f(0)^k\Big)  \sum_{j=n}^{\infty} \frac{z^j}{f(0)^{j}} \\ &=& f(f(0))\sum_{j=n}^{\infty} \frac{z^j}{f(0)^{j}}.
\eeq  
This together with \eqref{equi-h0-h1} shows that $h_0 \in \mathscr H_\kappa$ if and only if either $f(f(0))=0$ or $\frac{1}{f(0)-z} \in \mathscr H_\kappa.$
One may now apply Theorem~\ref{conse-proof}.
\end{proof}
\begin{remark}
\label{h0-h1-frac}
Assume that $f$ is a polynomial such that $f(0) \neq 0$ and $f(f(0)) \neq 0.$
By \eqref{equi-h0-h1} and \eqref{h1-equi},  
$h_0 \in \mathscr H_\kappa$ if and only if $\frac{1}{f(0)-z} \in \mathscr H_\kappa.$
\end{remark}

To complete the proof of Theorem~\ref{main}, we need another fact from the spectral theory. 

%

\begin{proof}[Proof of Theorem~\ref{main}]
In view of the discussion prior to Lemma~\ref{n-power}, we may assume that $T=\mathscr M_z$ is acting on a functional Hilbert space $\mathscr H_\kappa.$ 
Since $\mathscr M_z$ is analytic and injective, $\sigma_p(\mathscr M_z)=\emptyset.$
Hence, by Proposition~\ref{specrtum of perturbation}, $$\sigma_l(\mathscr M_z+f\otimes g)=\sigma_l(\mathscr M_z)\cup \sigma_p(\mathscr M_z+f\otimes g).$$
Consider the following cases$:$
\begin{enumerate}
\item[$\bullet$] If $\inp{f}{g}=0,$ then by  Lemma~\ref{Soumitra}(i),  $\mathscr M_z + f \otimes g$ is analytic, and hence $\sigma_p(\mathscr M_z + f \otimes g) \subseteq \{0\}=\{\inp{f}{g}\}.$  
\item[$\bullet$] If $\inp{f}{g} \neq 0,$ then by Proposition~\ref{kernel-one}(ii) and Lemma~\ref{Soumitra},, $\sigma_p(\mathscr M_z + f \otimes g) \subseteq \{\inp{f}{g}\}$ (since $0 \notin \sigma_l(\mathscr M_z)$).
\end{enumerate}
All  conclusions now follow from (i) and (ii) of Proposition~\ref{specrtum of perturbation}.
\end{proof}

The following is immediate from Theorem~\ref{main}.
\begin{corollary} Under the hypotheses of Theorem~\ref{conse-proof},  if $\mathscr M_z + f \otimes 1$ is analytic and $f(0) \neq 0,$  then  
\beqn
 \sigma_l(\mathscr M_z + f \otimes 1)= \sigma_l(\mathscr M_z), \quad r(\mathscr M_z + f \otimes 1) = r(\mathscr M_z).
 \eeqn
\end{corollary}



 Following \cite{Shimorin-2001}, we call the operator
$T':=T(T^*T)^{-1}$ the {\it Cauchy dual} of $T$ whenever $T$
is left-invertible. Here is an application to operators satisfying the so-called kernel condition (see \cite[Section~2]{ACJS}).
\begin{corollary} Assume  the hypotheses of Theorem~\ref{conse-proof}. Assume that $\mathscr M_z$ satisfies the kernel condition: 
\beqn
\mathscr M^*_z \mathscr M_z (\ker \mathscr M^*_z) \subseteq \ker \mathscr M^*_z.
\eeqn
If $\mathscr M_z$ has the wandering subspace property, then the operator $\mathscr M_z + 1 \otimes 1$ also has the wandering subspace property.
\end{corollary}
\begin{proof}
Recall the fact that for any left-invertible operator $S$ on $\mathcal H,$ $S$ is analytic if and only if the Cauchy dual $S'$ of $S$ has the wandering subspace property (see \cite[Corollary~2.8]{Shimorin-2001}). In view of this fact, it suffices to check that if $\mathscr M'_z$ is analytic, then so is $(\mathscr M_z + 1 \otimes 1)'.$

Let $T$ be a left invertible operator and $S=T + f \otimes g$ with $f \in \ker T^*$ of unit norm and $g \in \mathcal H.$ Note that 
\beqn
S^*S = (T^* + g \otimes f)(T+f \otimes g)  = T^*T + g \otimes g.
\eeqn
It is easy to see using $T'^*T'=(T^*T)^{-1}$ that 
\beqn
(S^*S)^{-1} = (T^*T)^{-1} - (1+\|T'g\|^2)^{-1}(T^*T)^{-1}g \otimes (T^*T)^{-1}g.
\eeqn
One may now verify that $S'$ is given by 
\beq \label{formula-Soumitra}
(T+ f \otimes g)' = T' + (1+\|T'g\|^2)^{-1} (f- T'g)\otimes (T^*T)^{-1} g.
\eeq
Since $\mathscr M_z$ satisfies the kernel condition, by \cite[Proposition 2.1]{ACJS}, $$(\mathscr M^*_z\mathscr M_z)^{-1} \ker \mathscr M^*_z=\ker \mathscr M^*_z.$$
Thus $\mathscr M'_z1=\beta z$ for some scalar $\beta.$
Apply now the formula \eqref{formula-Soumitra} to $f=g=1$ to conclude that $(\mathscr M_z + 1 \otimes 1)'$ is of the form $\mathscr M'_z + (\alpha  - \beta z) \otimes 1$ for some constants $\alpha, \beta.$ The desired conclusion now follows from Theorem~\ref{conse-proof}.
\end{proof}

\section{An example}

We conclude this note with one motivating example illustrating the general picture of the invariance of the left spectra.
\begin{example} 
Let $\mathscr H_\kappa$ be a functional Hilbert space such that all complex polynomials  are dense in $\mathscr H_\kappa.$  
Let $\mathscr M_z$ denote the operator of multiplication by $z$ on $\mathscr H_\kappa.$ 
For scalars $a, b \in \mathbb C,$ let $f(z)=az+b,$ $z \in \mathbb D.$ 
By Corollary~\ref{coro-main}, $\mathscr M_z + f \otimes 1$ is analytic if and only if 
$b=0$ or 
\beqn
\mbox{$b \neq 0,$ $a \neq -1$ and $\frac{1}{b-z} \notin \mathscr H_\kappa.$}
\eeqn
Assume that $\mathscr M_z$ is left invertible. 
Moreover, by Theorem~\ref{main}, 
\beq
\label{left-s-exm}
 \left.
 \begin{array}{cc}
 \sigma_l(\mathscr M_z + f \otimes 1) \backslash \{b\}  =  \sigma_l(\mathscr M_z) \backslash \{b\}, \\   r(\mathscr M_z + f \otimes 1)  = \max\{r(\mathscr M_z), |b|\}.
 \end{array}
\right\}
\eeq
Furthermore, we have the following$:$
\begin{enumerate}
\item[$\bullet$] Assume that $b=0$ and $a \neq 0.$ Then $f=az$ and 
\beqn
(\mathscr M_z + f\otimes 1)(1) = (1+a)z .
\eeqn
Thus $1 \in \ker(\mathscr M_z + f\otimes 1)$ if and only if $a=-1.$ Hence, 
by Proposition~\ref{kernel-one}, $\ker(\mathscr M_z  -z \otimes 1)$ is spanned by $\{1\}.$
In this case, 
$\mathscr M_z  -z \otimes 1$ is analytic and 
$$\sigma_l(\mathscr M_z  -z \otimes 1)=\sigma_l(\mathscr M_z) \sqcup \{0\}, \quad r(\mathscr M_z  -z \otimes 1)=r(\mathscr M_z).$$ 
\item[$\bullet$]  Assume that $b \neq 0,$ $a \neq -1$ and $\frac{1}{b-z} \in \mathscr H_\kappa$ (so that $|b| \Ge 1$). Then by Lemma~\ref{Soumitra}(iii) and Remark \ref{h0-h1-frac}, $b \in \sigma_p(\mathscr M_z + f \otimes 1),$ and hence by \eqref{left-s-exm}, we obtain 
\beqn
\sigma_l(\mathscr M_z + f \otimes 1) =   \sigma_l(\mathscr M_z) \cup \{b\}, \quad r(\mathscr M_z  + f \otimes 1)=\max\{r(\mathscr M_z), |b|\}.
\eeqn 
\item[$\bullet$] Assume that $b \neq 0,$ $a \neq -1$ and $\frac{1}{b-z} \notin \mathscr H_\kappa.$ Then, 
by Lemma~\ref{Soumitra}(ii) and Remark \ref{h0-h1-frac}, $\mathscr M_z + f \otimes 1$ is analytic. Also, by Proposition~\ref{kernel-one}(ii), $0$ does not belong to $\sigma_p(\mathscr M_z + f \otimes 1).$ Since the point spectrum of an analytic operator is contained in $\{0\},$  $\sigma_p(\mathscr M_z + f \otimes 1) = \emptyset.$ It now follows from \eqref{left-s-exm} that 
\beqn
\sigma_l(\mathscr M_z + f \otimes 1) =   \sigma_l(\mathscr M_z), \quad r(\mathscr M_z  + f \otimes 1)=r(\mathscr M_z).   
\eeqn 
\end{enumerate} 
It is evident that the above discussion extends, with suitable modifications, to the case when $f$ is a polynomial.
\hfill $\diamondsuit$
\end{example}


%

{}

\end{document}